\def\mR{\mathbb{R}}
 
\def\mN{\mathbb{N}}

\def\mC{\mathbb{C}}
\def\C{\mathrm{C}}

\def\R{\mathrm{R}}

\documentclass[10pt,draft,reqno]{amsart}
     \makeatletter
     \def\section{\@startsection{section}{1}%
     \z@{.7\linespacing\@plus\linespacing}{.5\linespacing}%
     {\bfseries
     \centering
     }}
     \def\@secnumfont{\bfseries}
     \makeatother
\theoremstyle{definition}

\theoremstyle{remark}

\numberwithin{equation}{section}
\newtheorem{Thm}{Theorem}[section]
\newtheorem{Def}{Definition}[section]
\newtheorem{Lem}{Lemma}[section]
\newtheorem{Pro}{Proposition}[section]
\newtheorem{Hyp}{Hypothesis}[section]

\newtheorem{Rem}{Remark}[section]

\numberwithin{equation}{section}
\begin{document}

\title[Stochastic Analysis of Nonlinear Wave Equations]{Stochastic Analysis and White Noise Calculus of Nonlinear Wave Equations with Application to Laser Generation and Propagation}
\author[S. S. Sritharan]{S. S. Sritharan*}
\address{S. S. Sritharan: NRC-Senior Research Fellow, National Academies of Science, Engineering and Medicine,  U. S. Air Force Research Laboratory, Wright Patterson Air Force Base, Ohio 45433, U. S. A.}
\thanks{* This research has been supported by the U. S. Air Force Research Laboratory through the National Research Council Senior Research Fellowship of the National Academies of Science, Engineering and Medicine}
\email{Provostsritharan@gmail.com}
\urladdr{https://www.linkedin.com/in/dr-sritharan/}

\author[S. Mudaliar]{Saba Mudaliar}
\address{Saba Mudaliar: U. S. Air Force Research Laboratory, Wright Patterson Air Force Base, Ohio 45433, U. S. A}
\email{saba.mudaliar@us.af.mil}

\subjclass[2010] {Primary 81S20, 81V10, 60G51; Secondary 60H15, 60H17}

\keywords{White noise calculus, semilinear stochastic evolutions, semigroup methods, Maxwell-Dirac equations, Zakharov system, stochastic Schr\"odinger equation, nonlinear schrodinger equation, stochastic quantization.}

\maketitle

\begin{abstract}
	In this paper we study a large class of nonlinear stochastic wave equations that arise in laser generation models and models for propagation in random media in a unified mathematical framework. Continuous and pulse-wave propagation models, free electron laser generation models, as well as laser-plasma interaction models have been cast in a convenient and unified abstract framework as semilinear evolution equations in a Hilbert space to enable stochastic analysis.  We formulate It\^o calculus and white noise calculus methods of treating stochastic terms and prove existence and uniqueness of mild solutions. 
\end{abstract}

\section{Introduction}
Mathematical study of high energy laser (HEL) and high power microwave (HPM) generation and propagation characteristics have numerous applications in engineering sciences \cite{Zohuri2016}. This subject has inspired extensive research in applied mathematics as well. In this paper we will undertake a comprehensive mathematical study of a number of models that arise in this subject using techniques of stochastic analysis and white noise calculus. The general philosophy of stochastic quantization for quantum field was proposed by \cite{Parisi1981} has also influenced this paper. Our study will include nonlinear Schr\"odinger equations and nonlinear Klein-Gordon equations modeling respectively continuous and pulse wave propagation in the atmosphere \cite{Sprangle2003}, Maxwell-Dirac equations modeling laser generation model for the free electron laser (FEL)\cite{Madey2010} and also Zakharov system modeling laser propagation through plasma \cite{Sulem1999}. It turns out that these nonlinear wave equations from such diverse applications can be cast as semilinear evolution in a Hilbert space $H$. We refer to \cite{Reed1976, Strauss1978, Strauss1989} for comprehensive study of these class of semilinear wave equations in the deterministic setting establishing solvability theorems and scattering theory. In forthcoming papers \cite{Sritharan2025a, Sritharan2025b} we study various convergent approximations and nonlinear filtering of the models studied in this paper. We also would like to point out that the approach taken in this paper is also very much in the spirit of \cite{Ouerdiane1998}.

We will consider a general class of stochastic semilinear evolution of the type:

\begin{equation}
	\frac{\partial \varphi}{\partial t} = \left(- iA +  V(t)\right )\varphi +J(\varphi), \label{eqn:1.1}
\end{equation}
where $A$ is a self-adjoint unbounded operator, $V$ is a space-time random field, and $J(\cdot)$ is a nonlinear interaction function. We will formulate the mild form of the above semilinear stochastic evolution in three levels of stochastic quantization with the random potential $V$ represented by stochastic, white noise and quantum stochastic process respectively:

(i)  It\^o calculus formulation:
\begin{equation}
	\varphi(t)= e^{-i A t}\varphi(0) +\int_{0}^{t}e^{-i A(t-s)}	 J(\varphi(s)) ds+\int_{0}^{t}e^{-i A(t-s)}\varphi(s) dW(s), \label{eqn:1.2}
\end{equation}
where $W$ is a $H$-valued Wiener process with covariance operator $Q$ (see \cite{DaPrato1996, DaPrato2014}).

(ii) White noise calculus formulation with Wick products \cite{Sritharan2023}:
\begin{equation}
	\varphi(t)= e^{-i A t}\varphi(0) +\int_{0}^{t}e^{-i A(t-s)}	 :J(\varphi(s)): ds+\int_{0}^{t}e^{-i A(t-s)}Z(s):\varphi(s)ds,\label{eqn:1.3}
\end{equation}
where $Z \in (E)^{*}$ is a white noise distribution \cite{Obata1994,Obata1994a, Kuo1996, Ji2021} in the Fock space Gelfand triple as defined later in this paper.
Here the symbol $":"$ denotes Wick product \cite{Kuo1996} and  $:J(\varphi):$ is the wick quantization of the nonlinearity, both will be defined later.

(iii) White noise calculus formulation with Wick products and "white noise operator"-valued noise $\Xi \in {\mathcal L}((E)\otimes H,(E)^{*}\otimes H)$ (quantum stochastic process) in the Fock space Gelfand triple \cite{Obata1994, Kuo1996, Ji2021}:
\begin{equation}
	\varphi(t)= e^{-i A t}\varphi(0) +\int_{0}^{t}e^{-i A(t-s)}	 J(\varphi(s))^{\diamond} ds+\int_{0}^{t}e^{-i A(t-s)}\Xi \diamond \varphi(s) ds.\label{eqn:1.4}
\end{equation}

Here the symbol $"\diamond"$ denotes operator Wick product \cite{Obata1999} and  $J(\varphi)^{\diamond}$ is the operator wick quantization of the nonlinearity, both will be defined later.

\section{Unified Abstract Nonlinear Wave Equation-I: Deterministic Medium}

In this section we will describe the mathematical structure of operators $A$ and $J$ and will indicate that these properties are satisfied by the five specific laser generation and propagation models identified later in the paper. We will also prove a solvability theorem for a deterministic semilinear evolution equation associated with (\ref{eqn:1.1}). This solvability theorem and analyticity theorem are needed for the white noise treatment of (\ref{eqn:1.3}) and (\ref{eqn:1.4}) since the $S$-transform leads to such a deterministic evolution and the characterization theorem for the inverse $S$-transform requires such a result.

\subsection{The Free Propagator and its $m$-Dissipative Generator}	
Let us begin with an important class of operators that cover the linear part of the equation \ref{eqn:1.1}. Although the class of nonlinear dissipative operators (or accretive operators) has led to powerful mathematical developments in nonlinear evolution equations \cite{Barbu1976}, we will only consider the linear case here. A linear operator $A$ with domain $D(A)$ in a Banach space $X$ is called dissipative \cite{Hille1996} if 
\begin{displaymath}
	\|u-\lambda Au\|\geq \|u\|, \forall u\in D(A)\mbox{ and } \forall \lambda >0.	
\end{displaymath}
Let $X$ be a Hilbert space. It can be shown that a linear operator $A$ in $X$ is dissipative if and only if
\begin{displaymath}
	\langle Au, u\rangle \leq 0, \hspace{.1in} \forall u\in D(A).
\end{displaymath}	
\begin{Def}
	An operator $A$ in $X$ is $m$-dissipative if
	
	(i)  $A$ is dissipative,
	
	(ii) $\forall \lambda >0$ and $\forall f\in X$ there exists $u\in D(A)$ such that $u-\lambda Au=f$.
\end{Def}
We will collect some relevant facts on $m$-dissipative operators on a Hilbert space in the theorem below (see also \cite{Cazenave1998}):	
\begin{Thm} \label{Thm:2.1}
	\begin{enumerate}
		\item If $A$ is $m$-dissipative in $X$ then $D(A)$ is dense in $X$.
		\item	If $A$ is a skew-adjoint operator in $X$ then $A$ and $-A$ are $m$-dissipative. 
		\item	If $A$ is a self-adjoint operator and $A\leq 0$, i.e., $\langle Au, u\rangle \leq 0, \forall u\in D(A)$ then $A$ is $m-dissipative$.
	\end{enumerate}
\end{Thm}

We recall here that if $A$ is self-adjoint then $iA$ is skew-adjoint:
\begin{displaymath}	
	(iA)^{*}=-iA^{*}=-iA.
\end{displaymath}		

\begin{Rem}
	The following examples of skew-ajoint operators are encountered later in section 4.		
	
	(1) Free Schrodinger operator:
	\begin{displaymath}
		A=i\Delta.
	\end{displaymath}
	Here $\Delta$ can be defined as a self-adjoint operator with a suitable choice of domain and hence $i\Delta$ is a skew-adjoint operator and is $m$-dissipative.
	
	(2) Wave operator: let $B=\sqrt{ -\Delta +k_{0}^{2}I}$, $D(B)=\left \{u \in L^{2}(\mR^{3}); Bu\in L^{2}(\mR^{3})\right \}$, and we define 
	\begin{displaymath}
		A= i\begin{bmatrix}
			0 & I\\
			-B^{2} & 0
		\end{bmatrix}.
	\end{displaymath}
	Then $A$ is a self-adjoint operator and $iA$ is skew-adjoint and $m$-dissipative.
	
	Moreover, the propagator associated with $iA$ is given by \cite{Segal1963,Reed1976}:
	\begin{displaymath}
		e^{-itA}=\begin{bmatrix}
			cos (tB) & B^{-1}sin (tB)\\
			-B sin(tB) & cos (tB)
		\end{bmatrix}.	
	\end{displaymath}
	
	(3) Dirac operator
	\begin{displaymath}
		D_{e}=-i \alpha\cdot\nabla + m\beta.	
	\end{displaymath}
	To understand this operator we describe it componentwise by setting $D_{e}u=v \in \mC^{4}$. Here the four dimensional complex vector space $\mC^{4}$ is called spinor space.
	\begin{displaymath}
		v_{j}=(D_{e}u)_{j}=i^{-1} \sum_{l=1}^{3}\sum_{k=1}^{4}(\alpha_{l})_{jk}\frac{\partial u_{k}}{\partial x_{l}}(x) +\sum_{k=1}^{4}\beta_{jk}u_{k}(x),
	\end{displaymath}
	where $\alpha_{k}$ and $\beta$ are Hermitian matrices satisfying the commutation relations
	\begin{displaymath}
		\alpha_{j}\alpha_{k}+\alpha_{k}\alpha_{j}=2\delta_{jk}I, \hspace{.1in} j, k=1,2,3,4.
	\end{displaymath}
	Here $\alpha_{4}=\beta$, and $I$ is the $4\times 4$ unit matrix.
	
	$D_{e}$ is a self-adjoint operator \cite{Kato1976} and hence $iD_{e}$ is a skew-adjoint operator and is $m$-dissipative.

	(4) Maxwell-Dirac operator
	\begin{displaymath}
		A= \begin{bmatrix}
			D_{e} & 0\\
			0 & A_{W}
		\end{bmatrix}
	\end{displaymath}	
	with $	D_{e}$ the Dirac operator and $A_{W}$ the wave operator defined above. Then $A$ is a self-adjoint operator and $iA$ is skew-adjoint and $m$-dissipative.
	
	(5) Zakharov operator
	\begin{displaymath}	
		A=\begin{bmatrix}
			\Delta & 0\\
			0 & \vert \nabla \vert
		\end{bmatrix}
	\end{displaymath}
	is a self-adjoint operator and hence $iA$ is skew-adjoint and is $m$-dissipative.
\end{Rem}

Let us now recall three well-known theorems that are relevant in characterizing the free propagator semigroup (actually group) generated by the linear part $iA$ and its bounded perturbations that we encounter in this paper \cite{Hille1996}.

\begin{Thm}\cite{Yosida1991}\label{Thm:2.2} (Hille-Yosida-Phillips) A linear operator $A$ is the generator of a contraction semigroup $S(t)$ in X if and only if $A$ is $m$-dissipative with a dense domain $D(A)$.
\end{Thm}
If $A$ is a skew-adjoint operator then $S(t)$ can be extendable to a group that is unitary.

\begin{Thm}\label{Thm:2.3} \cite{Kato1976} (Stone's Theorem)  Let ${\mathcal N}$ be a von Neumann algebra and let $\{U_{t}\}_{t\in \R} \subset {\mathcal N}$ be a group of unitary operators that is strongly continuous. Then there is a unique self-adjoint operator $A$ affiliated to ${\mathcal N}$ (i.e., $(A+iI)^{-1}\in {\mathcal N}$), the Stone generator, such that $U_{t}= e^{itA}$.
\end{Thm}	

\begin{Thm}\cite{Pazy1983} \label{Thm:2.4}
	Let $X$ be a Banach space and let $A$ be the infinitesimal generator of a $C_{0}$ semigroup denoted by $e^{At}$ on $X$ satisfying $\|S(t)\|\leq Me^{\omega t}$. If $B$ is a bounded linear operator on $X$ 	then $A+B$ is the infinitesimal generator of a $C_{0}$ semigroup denoted by $e^{(A+B)t}$ on $X$ satisfying $\|e^{(A+B)t}\|\leq Me^{(\omega +M\|B\|) t}$.
\end{Thm}

\subsection{Continuity Properties of the Nonlinearity}	
Let us define a class of nonlinearities and their continuity properties that would enable the solvability theorems of deterministic and stochastic type. Leter we will demonstrate that this class covers the five laser generation and interaction models considered in this paper.
\begin{Hyp} \label{Hyp:2.1}
	Let $A$ be a self-adjoint operator on a Hilbert space $H$ and $N$ be a positive integer. Suppose that $J$ is a densely defined nonlinear mapping on $H$ so that $J:D(A^{j})\rightarrow D(A^{j})$ for $0 \leq j\leq N$,:	
	\begin{enumerate}
		\item $\|A^{j}J(\varphi)\|\leq C \left (\|\varphi\|,\cdots, \|A^{j}\varphi\|\right) \|A^{j}\varphi\|,$
		\item $\|A^{j}(J(\varphi)-J(\psi))\|\leq C \left (\|\varphi\|, \|\psi\|,\cdots, \|A^{j}\varphi\|,\|A^{j}\psi\|\right) \|A^{j}(\varphi-\psi)\|,$
	\end{enumerate}
	for $j=0,1,\cdots,N$ and all $\phi, \psi\in D(A^{N})$, with constants $C(\cdot)$ monotonically increasing functions of the norms indicated.
\end{Hyp}	
\subsection{Local and Global Solvability Theory for the Unified Abstract Semilinear Equation}
We first prove solvability results for certain deterministic abstract semilinear evolution in the mild form. This result is needed in treating the $S$-transformed problem that arises in the white noise calculus analysis of the stochastic problem. We consider:
\begin{equation}
	\varphi(t)= e^{i A t}\varphi(0) +i\int_{0}^{t}e^{-i A(t-s)}	 J(\varphi(s)) ds +i \int_{0}^{t}e^{-i A(t-s)}\Theta(\zeta+z\eta)\varphi(s)  ds, \label{eqn:2.1}
\end{equation}
which corresponds to the strong form
\begin{equation}
	\frac{\partial \varphi}{\partial t} = - iA \varphi +J(\varphi)+\Theta(\zeta+z\eta) \varphi , \label{eqn:2.2}
\end{equation}
\begin{equation}
	\varphi(0)=\varphi_{0}. \label{eqn:2.3}
\end{equation}
Here $\Theta( z\zeta+\eta)$ is a holomorphic function in $z\in \mC$ and real analytic in $\zeta, \eta\in E$, where $E$ is a test function space.

\begin{Thm} \label{Thm:2.5}
	Let $A$ be a self-adjoint operator on a Hilbert space $H$ with domain $D(A)\subset H$ and let $J:D(A^{j})\rightarrow D(A^{j})$ for $0\leq j\leq N$, be a nonlinear mapping which satisfies Hypothesis \ref{Hyp:2.1}. Then, given $\varphi_{0}\in D(A^{N})$, there is a $T>0$ and a unique $D(A^{N})$-valued function $\varphi(t)$ on $[0,T)$ which satisfies (\ref{eqn:2.2}), (\ref{eqn:2.3}). For each set of the form $\left \{\varphi\vert \|A^{j}\varphi\|\leq a_{j}, j=0, \cdots, N\right \}$, $T$ can be chosen uniformly for all $\varphi_{0}$ in the set. The solution satisfies the estimate:
	\begin{equation}
		\sup_{t\in [0,T)} \sum_{j=0}^{N}\|A^{j}\varphi(t)\| \leq \sum_{j=0}^{N}	\|A^{j}\varphi_{0}\|e^{t C(T)|z|}, \hspace{.1in} z\in \mC.
	\end{equation}
\end{Thm}

\begin{proof} This theorem can be proved essentially following \cite{Reed1975, Reed1976} as the extra term $\Theta(\zeta+z\eta) \varphi$ we have in equation (\ref{eqn:2.1}) or (\ref{eqn:2.2}) can be absorbed in to a new definition of $J$ or can be handled separately. We will however sketch some of the essential steps of the proof.
	
	Since $A^{j}$ are closed operators, $D(A^{N})$ is a Banach space under the graph norm $\sum_{j=0}^{N}\|A^{j}\varphi\|$. For each $T>0$, let $X_{T}$ denote the set of continuous $D(A^{N})$-valued function $\varphi(\cdot)$ in the interval $[0,T)$ such that
	\begin{displaymath}
		\|\varphi(\cdot)\|_{X_{T}}:= \sup_{t\in [0,T)} \sum_{j=0}^{N}\|A^{j}\varphi(t)\|<\infty.
	\end{displaymath}
	For a fixed $\epsilon>0$, and given $\varphi_{0}\in D(A^{N})$ let
	\begin{displaymath}
		X_{T,\epsilon, \varphi_{0}}:=\left \{ \varphi(\cdot) \vert \varphi(0)=\varphi_{0},\hspace{.1in} \|\varphi(\cdot) -e^{-iA\cdot }\varphi_{0}\|_{X_{T}}\leq \epsilon \right \}.
	\end{displaymath}
	The goal would be to show that for small enough time $T$ the map
	\begin{displaymath}
		(	{\mathcal J}\varphi)(t)=e^{i A t}\varphi(0) +i\int_{0}^{t}e^{-i A(t-s)}	 J(\varphi(s)) ds +i \int_{0}^{t}e^{-i A(t-s)}\Theta(\zeta+z\eta)\varphi(s)  ds
	\end{displaymath}
	is a contraction on $X_{T,\epsilon, \varphi_{0}}$.
	
	Now using the properties of $J$ given in Hypothesis \ref{Hyp:2.1}, noting that the function $\Theta$ is bounded and the fact that the propagator $e^{-At}$ is a bounded operator in $H$ we can estimate
	\begin{displaymath}
		\|({\mathcal J}\varphi)(t)-e^{-iAt}\varphi_{0}\|\leq C_{\epsilon}T \sup_{t\in [0,T)}\|\varphi(t)\|,
	\end{displaymath}
	and in fact
	\begin{displaymath}
		\|A^{j}({\mathcal J}\varphi)(t)-e^{-iAt}A^{j}\varphi_{0}\|\leq C_{\epsilon}T \sup_{t\in [0,T)}\|A^{j}\varphi(t)\|,\hspace{.1in} 0\leq j\leq N.
	\end{displaymath}
	Summing over $j$ and taking supremum over $t$ we conclude that for small enought $T$, the map ${\mathcal J}$ is a contraction on $X_{T,\epsilon, \varphi_{0}}$ and hence has a unique fixed point $\varphi(\cdot)$ in $X_{T,\epsilon, \varphi_{0}}$. Strong differentiability of $\varphi(t)$ can be proven as in \cite{Reed1975}.
\end{proof}

The proof of Theorem \ref{Thm:2.6} and Theorem \ref{Thm:2.7} below can be essentially as in \cite{Reed1975} as the extra term $\varphi(t) \Theta $ can be either absorbed in to a new definition of $J$ with same estimates or treated separately.

\begin{Thm} \label{Thm:2.6} Suppose that in addition to the Hypothesis \ref{Hyp:2.1}, $J$ satisfies the following property: If $\varphi(t)$ is a $j$ times strongly differentiable $D(A^{N})$-valued function so that $\varphi^{(k)}(t)\in D(A^{N-k})$ and $A^{N-k}\varphi^{(k)}(t)$ is continuous for all $k\leq j$, then $J(\varphi(t))$ is $j$ times continuously differentiable, $(\frac{d}{dt})^{j}J(\varphi(t))\in D(A^{N-j-1})$ and $A^{N-j-1}(\frac{d}{dt})^{j}J(\varphi(t))$ is continuous. If $J$ satisfies this property then the solution of Theorem \ref{Thm:2.5} is $N$ times strongly continuously differentiable and $\varphi^{(j)}(t)\in D(A^{N-j})$ for all $j\leq N$.
\end{Thm}

\begin{Thm} \label{Thm:2.7} If instead of the estimate (2) of Hypothesis  \ref{Hyp:2.1} we have the estimate for the nonlinearity:
	\begin{displaymath}
		\|A^{j}J(\varphi)\|\leq C\left (\|\varphi\|, \cdots, \|A^{j-1}\varphi\|\right ) \|A^{j}\varphi\|,	\hspace{.1in} 1\leq j\leq N.
	\end{displaymath}
	Suppose that on any finite interval $[0,T)$ where the strong solution $\varphi(t)$ of equation \ref{eqn:2.1} exists, $\|\varphi(t)\|$ is bounded. Then the strong solution exists globally in $t$ and $\varphi(t)\in D(A^{N})$ for all $t$.
\end{Thm}
\begin{proof}
	
	The proof of this theorem is again essentially same as \cite{Reed1975} despite of the extra term in our case. We repeat this proof here since the steps are also used in the stochastic case later.
	
	Let $[0,T_{0})$ be the maximal interval on which a strongly differentiable solution of (\ref{eqn:2.1}) with values in $D(A^{N})$ exists. If $T_{0}<\infty$ then $\sum_{j=0}^{N}\|A^{j}\varphi(t)\|$ must go to infinity as $t \uparrow T_{0}$. Otherwise we could continue the solution across $T_{0}$ by using the fact that the $T$ in Theorem \ref{Thm:2.5} can be chosen uniformly for $\|A^{j}\varphi_{0}\|$  in bounded sets. Thus, to conclude global existence we need only show that $\sum_{j=0}^{N}\|A^{j}\varphi(t)\|$ is bounded on any finite interval $[0,T)$ where the solution exists. 
	
	By hypothesis $\|\varphi(t)\|$ is bounded on $[0,T)$. Thus,
	\begin{displaymath}
		C_{1}=\sup_{t\in [0,T)}C(\|\varphi(t)\|)<\infty,
	\end{displaymath}
	therefore since $\varphi$ satisfies $\varphi(t)=({\mathcal J}\varphi)(t)$ we have
	\begin{displaymath}
		\|A\varphi(t)\|\leq \|A\varphi_{0}\| + \int_{0}^{t}\|AJ(\varphi(s))\|ds +\int_{0}^{t}\|A\Theta \varphi(s)\|ds
	\end{displaymath}
	\begin{displaymath}
		\leq \|A\varphi_{0}\| + C_{1} \int_{0}^{t}\|A\varphi(s)\|ds. 
	\end{displaymath}
	Thus by iteration $\|A\varphi(t)\|\leq \|A\varphi_{0}\| e^{C_{1}t}$ so that $\|A\varphi(t)\|$ is bounded on $[0,T)$. 
	
	Similarly, defining 
	\begin{displaymath}
		C_{2}=\sup_{t\in [0,T)}C(\|\varphi(t)\|, \|A\varphi(t)\|)<\infty,
	\end{displaymath}
	we obtain
	\begin{displaymath} 
		\|A^{2}\varphi(t)\|\leq \|A^{2}\varphi_{0}\| + \int_{0}^{t}\|A^{2}J(\varphi(s))\|ds +\int_{0}^{t}\|A^{2}\Theta \varphi(s)\|ds
	\end{displaymath}
	\begin{displaymath}
		\leq \|A^{2}\varphi_{0}\| + C_{2} \int_{0}^{t}\|A^{2}\varphi(s)\|ds. 
	\end{displaymath}
	Since $\|\varphi(t)\|$ and $\|A\varphi(t)\|$ are bounded on $[0,T)$, and hence $C_{2}<\infty$, we conclude that $\|A^{2}\varphi(t)\|$ is bounded on $[0,T)$.
	
	Continuing this manner we can conclude that $\|A^{j}\varphi(t)\|$ is bounded on $[0,T)$ for each $j\leq N$.
\end{proof}

\begin{Thm} \label{Thm:2.8}
	Let $A$ and $J$ satisfy the Hypothesis \ref{Hyp:2.1} for all $j$. Suppose that $J$ satisfies the hypothesis of Theorem \ref{Thm:2.7} for all $N$ that any strong solution of \ref{eqn:2.1} is a priori bounded on any finite interval. If $\varphi_{0}\in \bigcap_{j=1}^{\infty}D(A^{j})$, then the solution $\varphi(t)$ of \ref{eqn:2.1} is infinitely strongly differentiable in $t$ and each time derivative has values in $\bigcap_{j=1}^{\infty}D(A^{j})$.
\end{Thm}

We remark here that with further restrictions on $J$, we can obtain a local solvability theorem similar to Theorem \ref{Thm:2.5} for operator valued solution $\varphi(t)\in {\mathcal L}(H)$ strongly continuous in $t$, if $\varphi(0)\in {\mathcal L}(H)$ as in solvability theorems for infinite dimensional Riccati operator equations with unbounded operator coefficients. See Lemma 5.3.3 and Theorem 6.4.2 of \cite{Meyer1973} and also numerous other references cited later in the paper on infinite dimensional Riccati equations. We will explore and expand upon these solvability theorems for operator valued evolutions in future papers \cite{Sritharan2025a, Sritharan2025b}.

Recalling the definitions of Banach and operator valued holomorphic functions \cite{Hille1996}, we have the following results:
\begin{Thm}
	The maps $z\rightarrow \varphi(t,z)$ from $\mC \rightarrow H$ and as an operator valued map $z\rightarrow \varphi(t,z)$ from $\mC \rightarrow {\mathcal L}(H,H)$ are holomorphic and entire.
\end{Thm}	

{\bf Proof}: According to \cite{Hille1996}, Definition 3.10.1, need to show that $\varphi$ is weakly differentiable with respect to $z$: for the Hilbert space valued case (that is $\varphi(z)\in H$), one only needs to prove differentiability of $\langle \varphi(z),v\rangle, v\in H$ and in the operator valued case (that is $\Phi(z)\in {\mathcal L}(H)$), need to prove the differentiability of $\langle \Phi(z) u, v\rangle, u,v\in H$.  Hence we consider the evolution equation corresponding to the Gateaux derivative of $\varphi$ with respect to $z$ which we denote as $\partial_{z}\varphi:=\Phi$:
\begin{displaymath}
	\Phi(t)= i\int_{0}^{t}e^{-i A(t-s)} \left \{	 J'(\varphi(s)) +\Theta(\zeta+z\eta)\right \} \Phi(s)ds  
\end{displaymath}
\begin{displaymath}
 + i\int_{0}^{t}e^{-i A(t-s)}\partial_{z}\Theta(\zeta+z\eta)\varphi(s) ds, 
\end{displaymath}
which corresponds to the following linear partial differential equation (operator equation in Hilbert space):
\begin{equation}
	\frac{\partial \Phi}{\partial t} = \left\{- iA + J'(\varphi(t)) +\Theta(\zeta+z\eta)  \right \}\Phi +\partial_{z}\Theta(\zeta+z\eta)\varphi(t) , 
\end{equation}
\begin{equation}
	\Phi(0)=0. 
\end{equation}
Note that given the bounded function $ J'(\varphi(t)) +\Theta(\zeta+z\eta)$, by Theorem \ref{Thm:2.4}, $- iA + J'(\varphi(t)) +\Theta(\zeta+z\eta)$ generates a $C_{0}$ semigroup in $H$. This then by Theorem \ref{Thm:2.4} allows us to conclude that the non-homogeneous linear evolution (\ref{eqn:2.2})-(\ref{eqn:2.3}) with bounded forcing term $\varphi(t) \partial_{z}\Theta(\zeta+z\eta) $ has a unique bounded solution $\Phi\in C_{b}([0,T];H)$ by standard semigroup theory \cite{Kato1976a, Pazy1983}. In the case of $\Phi(t)$ being interpreted as an operator valued evolution with values in ${\mathcal L}(H)$ we refer to \cite{Temam1971, DaPrato1972, Tartar1974} for solvability and will explore further details in our forthcoming papers \cite{Sritharan2025a, Sritharan2025b}.

\section{Unified Abstract Nonlinear Wave Equation-II: Stochastic Medium}
\subsection{It\^o Calculus Method}
In this section we will treat the mild solution form \ref{eqn:1.2} using It\^o calculus method combined with method of stochastic semilinear evolution in the spirit of \cite{DaPrato1996, DaPrato2014, Mohan2017}. Let $(\Omega, \Sigma, \Sigma_{t},m)$ be a complete filtered probability space. Let $W(t)$ be an $H$-valued Wiener process with symmetric trace-class covariance operator $Q\in {\mathcal L}(H,H)$:
\begin{displaymath}
	{\mathcal E}[(W(t),\psi)(W(\tau),\phi)]=t\wedge \tau (Q\psi,\phi),\hspace{.1in} \forall \phi,\psi\in H,
\end{displaymath}
and $\mbox{ Tr}Q <+\infty$ (finite trace). Using the eigensystem $\{ \lambda_{i}, e_{i}\}_{i=1}^{\infty}, e_{i}\in H, i\geq 1$ of $Q$ we can also express the infinite dimensional Wiener process $W(t)$ as a series expansion \cite{DaPrato2014}:
\begin{displaymath}
	W(t)=\sum_{i=1}^{\infty}\sqrt{\lambda_{i}}e_{i}\beta_{i}(t) \in H, \hspace{.1in} t\geq 0,
\end{displaymath}
where $\beta_{i}$ are standard independent Brownian motions and $\mbox{ Tr}Q=\sum_{i=1}^{\infty} \lambda_{i} <+\infty.$ The main theorem in this section is proven using techniques similar to the paper \cite{Mohan2017}.

Let us now define a stopping time $\tau_{\Lambda}$ as:
\begin{displaymath}
	\tau_{\Lambda}:=\inf\left \{ t>0; \sup_{0\leq j\leq N-1}\|A^{j}\varphi(t)\| >\Lambda \right\}.
\end{displaymath}

\begin{Thm} \label{Thm:3.1}
	Suppose $A$ and $J$ satisfy the Hypothesis \ref{Hyp:2.1} and the hypotheses of Theorem \ref{Thm:2.7}. Then there is a stopping time $\tau_{\Lambda}(\omega)>0$ such that there is a unique $\Sigma_{t}$-adapted $H$-valued stochastic process $\varphi(t,\omega), t\in (0,\tau_{\Lambda}), \omega \in \Omega$ with bound
	\begin{displaymath}
		{\mathcal E} [	\sup_{t\in (0,\tau_{\Lambda})}\sum_{j=0}^{N}\|A^{j}\varphi(t)\|^{2}]\leq C(T, Q){\mathcal E}[\sum_{j=0}^{N} \|A^{j}\varphi_{0}\|^{2} ], 
	\end{displaymath}
	that satisfies 
	\begin{equation}
		\varphi(t)= e^{-i A t}\varphi_{0} +\int_{0}^{t}e^{-i A(t-s)}	 J(\varphi(s)) ds+\int_{0}^{t}e^{-i A(t-s)}\varphi(s) dW(s), \hspace{.1in} 0\leq t\leq \tau_{\Lambda},
	\end{equation}
	\begin{displaymath}
		\varphi(0)=\varphi_{0}\in H,
	\end{displaymath}
	with probability one in the time interval $(0,\tau_{\Lambda})$ and for a given $0 <\rho <1$ 
	\begin{displaymath}
		m\left \{\omega: \tau_{\Lambda}(\omega) >\rho \right\}\geq 1-\rho^{2}M,
	\end{displaymath}
	with $M$ depending on $\varphi_{0}, Q$ and independent of $\rho$. 
\end{Thm}

\begin{proof}
	For each $T>0$, let $X_{T}$ denote the set of continuous $D(A^{N})$-valued $\Sigma_{t}$-adapted stochastic processes $\varphi(\cdot)$ in the interval $[0,T)$ such that
	\begin{displaymath}
		\|\varphi(\cdot)\|_{X_{T}}^{2}:= {\mathcal E}\left [\sup_{t\in [0,T)} \sum_{j=0}^{N}\|A^{j}\varphi(t)\|^{2}\right ]<\infty.
	\end{displaymath}
	For a fixed $\epsilon>0$, and given $\varphi_{0}\in D(A^{N})$ let
	\begin{displaymath}
		X_{T,\epsilon, \varphi_{0}}:=\left \{ \varphi(\cdot) \vert \varphi(0)=\varphi_{0},\hspace{.1in} \|\varphi(\cdot) -e^{-iA\cdot }\varphi_{0}\|_{X_{T}}\leq \epsilon \right \}.
	\end{displaymath}
	The goal would be to show that for small enough time $T$ the map ${\mathcal J}:X_{T,\epsilon, \varphi_{0}}\rightarrow X_{T,\epsilon, \varphi_{0}}$:
	\begin{equation}
		({\mathcal J}\varphi)(t)= e^{-i A t}\varphi_{0} +\int_{0}^{t}e^{-i A(t-s)}	 J(\varphi(s)) ds+\int_{0}^{t}e^{-i A(t-s)}\varphi(s) dW(s)	
	\end{equation}
	is a contraction on $X_{T,\epsilon, \varphi_{0}}$.

	The following inequality is a consequence of the well-known stochastic convolution estimate \cite{DaPrato1996, DaPrato2014, vanNeerven2020}:
	\begin{displaymath}
		{\mathcal E}\left [ \sup_{0\leq s\leq t} \vert| \int_{0}^{s}e^{-iA(t-r)}\varphi(r)dW(r)\vert|^{2}\right]\leq K {\mathcal E}\left [\int_{0}^{t}\vert | \varphi(r)\sqrt{Q}\vert |_{HS}^{2}dr \right] 
	\end{displaymath}
	\begin{displaymath}
		\leq K (\mbox{tr}Q){\mathcal E}\left [\int_{0}^{t}\vert | \varphi(r)\vert |^{2}dr \right].
	\end{displaymath}
	In fact we have
	\begin{displaymath}
		{\mathcal E}\left [ \sup_{0\leq s\leq t} \vert| \int_{0}^{s}A^{j} e^{-iA(t-r)}\varphi(r)dW(r)\vert|^{2}\right]
	\end{displaymath}
\begin{displaymath}
\leq K {\mathcal E}\left [\int_{0}^{t}\vert | A^{j}\varphi(r)\sqrt{Q}\vert |_{HS}^{2}dr \right] 
	\end{displaymath}
	\begin{displaymath}
		\leq K (\mbox{tr}Q){\mathcal E}\left [\int_{0}^{t}\vert |A^{j} \varphi(r)\vert |^{2}dr \right]\leq K T (\mbox{tr}Q) {\mathcal E} \left [\sup_{0\leq s\leq t}\vert |A^{j} \varphi(t)\vert^{2}\right ].
	\end{displaymath}
	Hence
	\begin{displaymath}
		\sum_{j=0}^{N}  {\mathcal E}\left [ \sup_{0\leq s\leq t} \vert| \int_{0}^{s}A^{j} e^{-iA(t-r)}\varphi(r)dW(r)\vert|^{2}\right]  
	\end{displaymath}
\begin{displaymath}
\leq K T (\mbox{tr}Q)\sum_{j=0}^{N} {\mathcal E}\left [\sup_{0\leq s\leq t}\vert |A^{j} \varphi(s)\|^{2}\right].	
	\end{displaymath}
	Now noting that due to Hypothesis \ref{Hyp:2.1} and Hypothesis of \ref{Thm:2.7} on $J$,
	\begin{displaymath}
			{\mathcal E}\left [\|\int_{0}^{\tau_{\Lambda}}e^{-A(t-s)}A^{j}J(\varphi(s))ds\|^{2}\right ] \leq C(\Lambda,T ) {\mathcal E}\left [\sup_{0\leq t\leq \tau_{N}}\| A^{j}\varphi(t)\|^{2}\right ],
	\end{displaymath}
	we can arrive at
	\begin{displaymath}
		{\mathcal E}\left [\sup_{0\leq t\leq \tau_{\Lambda}}\sum_{j=0}^{N}	\|A^{j}({\mathcal J}\varphi)(t)-e^{-iAt}A^{j}\varphi_{0}\|^{2}\right ]
	\end{displaymath}
\begin{displaymath}
\leq C (\Lambda,\mbox{tr}Q) {\mathcal E}\left [\sup_{t\in [0,\tau_{\Lambda})}\sum_{j=0}^{N}\|A^{j}\varphi(t)\|^{2}\right].	
	\end{displaymath}
	This completes the fixed point argument. The tail probability estimate can be obtained by arguments similar to \cite{Mohan2018}.
\end{proof}

We also recall here that as an $H$-valued square integrable random variable $\varphi(\cdot)\in L^{2}(\Omega, H)$ we have an infinite dimensional multiple Wiener expansion \cite{Alshanskiy2014}:
\begin{displaymath}
	\varphi(t)=\sum_{0}^{\infty}I_{n}(f_{n}),\hspace{.1in} f_{n}\in \hat{L}^{2}((0,\tau)^{n}:{\mathcal L}_{2}^{n}(H^{\otimes n};H)), n\in \mN,
\end{displaymath}
with
\begin{displaymath}
	{\mathcal E}[\|\varphi\|]^{2} =\sum_{0}^{\infty}n! \|f_{n}\|^{2}_{L^{2}((0,\tau)^{n}:{\mathcal L}_{2}^{n}(H^{\otimes n};H))}.	
\end{displaymath}
where
\begin{displaymath}
	I_{n}(\zeta)=n! \int_{0}^{\tau}\int_{0}^{t_{n-1}}\cdots \int_{0}^{t_{1}}\zeta(t_{1},t_{2},\cdots, t_{n})dW(t_{1})\cdots dW(t_{n}),
\end{displaymath}
with $\zeta: [0,\tau]^{n}\times \Omega \rightarrow {\mathcal L}_{2}^{n}(H^{\otimes n};H).$ The multiple Wiener integrals satisfy the following orthogonality relations:
\begin{displaymath}
	{\mathcal E}[(I_{n}(f), I_{m}(g))_{H}]=
	\left\{
	\begin{array}{ll}
		0  & \mbox{if } n\neq m \\
		(n!)^{2}	(f,g)_{L^{2}((0,\tau)^{n}:{\mathcal L}_{2}^{n}(H^{\otimes n};H))} & \mbox{if } m=n.
	\end{array}
	\right.
\end{displaymath}

\subsection{White Noise Calculus Method}
This section relies on the general mathematical framework for the white noise calculus method and the Fock space Gelfand triple as described in \cite{Kuo1975, Hida1980, Hida1994, Obata1994}. White noise method for stochastic partial differential equations has been studied in \cite{Derm1996, Holden2010, Sritharan2023}.
We first $S$-transform the equation \ref{eqn:1.3}, use the deterministic solvability theory we established earlier for \ref{eqn:2.1} and then use the characterization theorem for the solution of the $S$-transformed equation to deduce the solvability theorem for equation \ref{eqn:1.3}.

Let ${\mathcal T}$ be a topological space with Borel measure $\nu(dt)$ denoted $dt$ for simplicity (here ${\mathcal T}$ can be taken as $\mR^{n}$ as well). Let ${\mathcal A}$ be a positive self-adjoint operator on ${\mathcal H}=L^{2}({\mathcal T}, \nu, \mR)$ with Hilbert-Schmidt inverse and $\mbox{ infSpec}({\mathcal A})>1$. We can then construct a Gelfand-triple in a well-known way \cite{Gelfand1964} as:
\begin{displaymath}
	E \subset {\mathcal H}=L^{2}({\mathcal T},\nu, \mR)\subset E^{*},
\end{displaymath}
where $E$ and $E^{*}$ are considered as respectively, spaces of test and generalized functions on ${\mathcal T}$ and we assume the Kubo-Tanenaka hypothesis (H1)-(H3) \cite{Kubo1980} to ensure that the delta function $\delta_{t}$ in the space $E^{*}$. In fact when $E={\mathcal S}(\mR^{n})$ the Schwartz space of rapidly decreasing test functions, ${\mathcal H}=L^{2}(\mR^{n})$ and $E^{*}={\mathcal S}'(\mR^{n})$ is the space of tempered distributions. 

In fact, setting $\vert\zeta\vert_{p}:=\vert{\mathcal A}^{p}\zeta\vert_{0}$ and denoting $E_{p}$ to be the completion with respect to the norm $\vert \cdot \vert_{p}$, we have 
\begin{displaymath}
	E \cong \mathop{proj lim}_{p\rightarrow \infty} E_{p}, \hspace{.1in} E^{*}\cong \mathop{ind lim}_{p\rightarrow \infty }E_{-p}.
\end{displaymath}
Let $\mu$ be the Gaussian measure on $E^{*}$ and consider the complex Hilbert space $(L^{2})=L^{2}(E^{*},\mu;\mC)$ is canonically isomorphic to the Boson Fock space over ${\mathcal H}_{\mC}$ through the Wiener-It\^o-Segal isomorphism:
\begin{displaymath}
	(L^{2})=L^{2}(E^{*},\mu;\mC)=\sum_{n=0}^{\infty}\oplus {\mathcal H}^{\hat{\otimes}n }_{\mC},
\end{displaymath}
each $\varphi \in (L^{2})$ admits a Wiener-It\^o expansion:
\begin{displaymath}
	\varphi(x)=\sum_{n=0}^{\infty} \langle :x^{\otimes n}:,f_{n}\rangle, \hspace{.1in} x\in E^{*}, f_{n}\in {\mathcal H}^{\hat{\otimes}n }_{\mC},
\end{displaymath}
with
\begin{displaymath}
	\|\varphi\|_{0}^{2}=\int_{E^{*}}\vert \varphi(x)\vert^{2}\mu(dx)=\sum_{0}^{\infty}n! \vert f_{n}\vert_{0}^{2}.
\end{displaymath}
Given $\zeta\in E_{\mC}$ we define the exponential vector as
\begin{displaymath}
	\varphi_{\zeta}(x)=\sum_{n=0}^{\infty}  \langle :x^{\otimes n}:,\frac{\zeta^{\otimes n}}{n!}\rangle =\exp\left ( \langle x, \zeta\rangle -\frac{1}{2} \langle \zeta, \zeta\rangle \right ), \hspace{.1in} x\in E^{*},
\end{displaymath}
with $\varphi_{0}$ being the Fock vacuum.

We define the second quantization operator $\Gamma({\mathcal A})$ of ${\mathcal A}$ as:
\begin{displaymath}
	\Gamma({\mathcal A})\varphi_{\zeta} =\varphi_{{\mathcal A}\zeta}=\sum_{n=0}^{\infty}  \langle :x^{\otimes n}:,\frac{({\mathcal A}\zeta)^{\otimes n}}{n!}\rangle 
\end{displaymath}
\begin{displaymath}
=\exp\left ( \langle x, {\mathcal A}\zeta\rangle -\frac{1}{2} \langle {\mathcal A} \zeta,{\mathcal A} \zeta\rangle \right ), \hspace{.1in} \zeta \in D({\mathcal A}), x\in E^{*}_{\mC}.
\end{displaymath}
The second quantization operator is unique, positive, self-adjoint with a Hilbert-Schmidt inverse \cite{Simon1974,Obata1994, Kuo1996}, we can construct a Gelfand-triple on $(L^{2})$ based on this operator:
\begin{displaymath}
	(E)\subset (L^{2}) =L^{2}(E^{*},\mu,\mC)\subset (E)^{*},
\end{displaymath}
where the elements in $(E)$ and $(E)^{*}$ are respectively white noise test functionals and white noise generalized functionals. The above Gelfand triple is called Hida-Kubo-Takenaka space \cite{Kubo1980}. We will now recall a generalization of this construction due to Kondratiev and Streit \cite{Kon1993}. 

Let $\beta$ be a fixed number with $0\leq \beta \leq 1$. For $\varphi\in L^{2}(E^{*},\mu)$ we introduce a new norm
\begin{displaymath}
	\|\varphi\|_{p,\beta}^{2}=\sum_{n=0}^{\infty} (n!)^{1+\beta}\vert f_{n}\vert^{2}_{p}, 
\end{displaymath}
\begin{displaymath}
	\varphi(x)=\sum_{n=0}^{\infty} \langle :x^{\otimes n}:,f_{n}\rangle, \hspace{.1in} x\in E^{*}, f_{n}\in {\mathcal H}^{\hat{\otimes}n }_{\mC}.
\end{displaymath}
For any $p\geq 0$, $(E_{p})_{\beta}=\left \{\varphi; \|\varphi\|_{p,\beta}<\infty\right\}$ becomes a Hilbert space. We set
\begin{displaymath}
	(E)_{\beta}=\mathop{projlim}_{p\rightarrow \infty} (E_{p})_{\beta},
\end{displaymath}
which becomes a countable nuclear Hilbert space. In fact \cite{Obata1999} for any $p\geq 0$ the canonical map $i_{p}: (E_{p+1})_{\beta} \rightarrow (E_{p})_{\beta}$ is of Hilbert-Schmidt type with $\|i_{p}\|_{HS}=\|\Gamma(A)^{-1}\|_{HS}$, where $\Gamma(A)$ is the second quantization operator of $A$ acting on $\Gamma({\mathcal H}_{\mC})$.

For $0\leq \beta <1$ and $p\geq 0$ we set
\begin{displaymath}
	\|\varphi\|^{2}_{-p-\beta}=\sum_{n=0}^{\infty} (n!)^{1-\beta}\vert f_{n}\vert^{2}_{-p}, \hspace{.1in} \varphi \sim (f_{n}).
\end{displaymath}
Then $\|\cdot\|_{-p-\beta}$ is a Hilbertian norm on $L^{2}(E^{*},\mu)$ and we denote its completion as $(E_{-p})_{-\beta}$. Hence the dual space of $(E)_{\beta}$ is obtained as 
\begin{displaymath}
	(E)^{*}_{\beta} \cong \mathop{indlim}_{p\rightarrow \infty} (E_{-p})_{-\beta},	
\end{displaymath}
and we obtain the Kondratiev-Streit type Gelfand triple
\begin{displaymath}
	(E)_{\beta}\subset L^{2}(E^{*},\mu)\subset (E)^{*}_{\beta}.
\end{displaymath}

For each $\Phi\in (E)^{*}$ there exists a unique sequence $\{F_{n}\}_{n=0}^{\infty}$, $F_{n}\in (E_{\mC}^{\otimes n})^{*}_{\mbox{sym}}$ such that
\begin{displaymath}
	\langle \langle \Phi, \varphi\rangle\rangle =\sum_{n=0}^{\infty}n!\langle F_{n},f_{n}\rangle, \hspace{,1in} \varphi \in (E),
\end{displaymath}
where the canonical duality pairing $\langle \langle \cdot, \cdot\rangle\rangle: (E)^{*}\times (E) \rightarrow \mC$. We also have the following formal representation for $\Phi\in (E)^{*}$:
\begin{displaymath}
	\Phi(x)=\sum_{n=0}^{\infty} \langle :x^{\otimes n}:, F_{n}\rangle, \hspace{.1in} x\in E^{*}, 	F_{n}\in (E_{\mC}^{\otimes n})^{*}_{\mbox{sym}}.
\end{displaymath}
The $S$-transform of a generalized white noise functions $\Phi\in (E)^{*}$ is a function on $E_{\mC}$ defined by the action of $\Phi$ on the exponential vector $\varphi_{\zeta}$:
\begin{displaymath}
	S\Phi(\zeta) =\langle \langle \Phi, \varphi_{\zeta}\rangle\rangle, \hspace{.1in} \zeta \in E_{\mC}.
\end{displaymath}
We thus also have a series representation:
\begin{displaymath}
	S\Phi(\zeta)=\sum_{n=0}^{\infty} \langle F_{n},\zeta^{\otimes n}\rangle, \hspace{.1in} \zeta \in E_{\mC}.
\end{displaymath}
It is well-known that \cite{Kubo1980, Obata1994, Kuo1996} for $\zeta, \eta \in E_{\mC}$ and $\Phi\in (E)^{*}$, 
\begin{displaymath} 
	z\rightarrow S\Phi(z \zeta +\eta)=\langle \langle \Phi, \varphi_{z \zeta +\eta}\rangle\rangle, \hspace{.1in} z\in \mC,
\end{displaymath}
is an entire holomorphic function. The characterization theorem \cite{Obata1994, Kuo1996} gives the reverse results:

\begin{Thm} \label{Thm:3.2} Let $F$ be a $\mC$-valued function on $E_{\mC}$. Then $F=S\Phi$ for some $\Phi\in (E)^{*}$ if and only if
	\begin{enumerate}
		\item for fixed $\zeta,\eta\in E_{\mC}$, the function $z\rightarrow F(z\zeta +\eta)$, $z\in \mC$ is an entire holomorphic function on $\mC$;
		\item there exists $C\geq 0, K\geq 0$ and $p\in \mR$ such that
		\begin{displaymath}
			\vert F(\zeta)\vert\leq C\exp{(K\vert \zeta\vert^{2}_{p})}, \hspace{.1in} \zeta \in E_{\mC},
		\end{displaymath}
		where $\vert \cdot \vert_{p}=\vert {\mathcal A}^{p}\cdot\vert_{0}$.
	\end{enumerate}
\end{Thm}

The $\Phi, \Psi\in (E)^{*}$ be two generalized white noise functionals with $F=S\Phi$ and $G=S\Psi$. Then $FG$ satisfies the conditions (1), (2) of the characterization theorem and hence there exists $\Theta\in (E)^{*}$ such that $S\Theta=FG$. Based on this observation we define Wick product of two distributions as follows:
\begin{Def}
	Let $\Phi, \Psi\in (E)^{*}$ be two generalized white noise functionals then the Wick product $\Phi : \Psi\in (E)^{*}$ is the unique generalized white noise functional such that $S(\Phi :\Psi)=(S\Phi)( S\Psi).$
\end{Def}
We define the Wick product of our nonlinearity $J(\varphi)$ using the following property:
\begin{displaymath}
	S (:J(\varphi):) = J(S \varphi).
\end{displaymath}
Now taking $S$ transform of \ref{eqn:1.3} we obtain \ref{eqn:2.1}. The first main theorem in this section is a consequence of Theorem \ref{Thm:2.5} and Theorem \ref{Thm:3.2}:
\begin{Thm} \label{Thm:3.3}
	For $\varphi_{0}\in H$ and small enough time $T$ there exists a unique solution $	\varphi\in ((E)\otimes H)^{*}$   to \ref{eqn:1.2} such that the following series converges in $ ((E)\otimes H)^{*}$:
	\begin{displaymath}
		\varphi(Z)=\sum_{0}^{\infty}\langle :Z^{\otimes n}:,F_{n}\rangle, \hspace{.1in} Z\in E^{*}, F_{n}\in (E_{\mC}^{\otimes n}\otimes H)^{*}_{\mbox{sym}}.
	\end{displaymath}
\end{Thm}

For the case of operator valued noise we use the characterization theorem for the operator symbols \cite{Obata1994} to deduce the solvability theorem for \ref{eqn:1.4} using iterative method for Wick quantized operator equations developed in \cite{Obata1999}.

Let us now introduce the Hida derivative \cite{Hida1980, Hida1994, Obata1994, Kuo1996}.  For any $y\in E^{*}$ and $\varphi \in (E)$ we set
\begin{displaymath}
	D_{y}\varphi(x)=\lim_{\theta\rightarrow 0}\frac{\varphi(x+\theta y)-\varphi(x)}{\theta}, \hspace{.1in} x\in E^{*}.
\end{displaymath}
The limit exists \cite{Obata1994} in ${\mathcal L}((E),(E))$. Since the delta function $\delta_{t}$ belongs to $E^{*}$ by Kubo-Takenaka hypotheses \cite{Kubo1980} we can specialize the Hida derivative to the case of $y=\delta_{t}$ and denote
\begin{displaymath}
	\partial_{t}:=D_{\delta_{t}}, \hspace{.1in} t\in {\mathcal T}.
\end{displaymath}
The operator $\partial_{t}\in {\mathcal L}((E),(E))$ is called the annihilation operator at a point $t\in {\mathcal T}$. The adjoint of this operator $\partial^{*}_{t}\in {\mathcal L}((E)^{*},(E)^{*})$ is called the creation operator. These two operators satisfy the canonical commutative relation (CCR):
\begin{displaymath}
	[\partial_{s}, \partial_{t}]=0, \hspace{.1in} [\partial_{s}^{*},\partial_{t}^{*}]=0, \hspace{.1in} [\partial_{s},\partial_{t}^{*}]=\delta_{s}(t)I, \hspace{.1in} s,t\in {\mathcal T}.
\end{displaymath}
For any $\kappa \in (E_{\mC}^{\otimes (l+m)})^{*}$ there exists a unique operator $\Xi_{l,m}(\kappa)\in {\mathcal L}((E),(E)^{*})$ such that
\begin{displaymath}
	\langle\langle \Xi_{l,m}(\kappa)\varphi, \psi\rangle \rangle =\langle \kappa, \langle \langle \partial^{*}_{s_{1}}\cdots \partial_{s_{l}}^{*}\partial_{t_{1}}\cdots \partial_{t_{m}}\varphi, \psi\rangle \rangle \rangle, \hspace{.1in} \varphi, \psi \in (E).
\end{displaymath}
where $\partial_{s}^{*}$ and $\partial_{t}$ are respectively creation and annihilation operators at points $s$ and $t$ in the field parameter space ${\mathcal T}$. The operator $\Xi_{l,m}$ is called the integral kernel operator and represented formally as:
\begin{displaymath}
	\Xi_{l,m}(\kappa)=\int_{{\mathcal T}^{l+m}}\kappa(s_{1},\cdots,s_{l},t_{1},\cdots,t_{m})\partial_{s_{1}}^{*}\cdots\partial_{s_{l}}^{*}\partial_{t_{1}}\cdots \partial_{t_{m}}ds_{1}\cdots ds_{l}dt_{1}\cdots dt_{m}.
\end{displaymath}
The following theorem \cite{Obata1993} gives a series expansion for Fock space operators. 

\begin{Thm}\label{Thm:3.4} For any $\Xi\in {\mathcal L}((E), (E)^{*})$ there is a unique family of kernel distributions $\kappa_{l,m}\in (E^{\otimes (l+m)})^{*}_{\mbox{sym $(l,m)$}}$ such that
	\begin{displaymath}
		\Xi\varphi =\sum_{l,m=0}^{\infty} \Xi_{l,m}(\kappa_{l,m})\varphi, \hspace{.1in} \varphi \in (E),
	\end{displaymath}
	converges in $(E)^{*}$.
\end{Thm}
The symbol \cite{Berezin1966,Berezin1971, Kree1978, Hida1994, Obata1993, Obata1994, Kuo1996} of a Fock space operator $\Xi\in {\mathcal L}((E),(E)^{*})$ is a function on $E_{\mC}\times E_{\mC}$ defined by
\begin{displaymath}
	\hat{\Xi}(\zeta, \eta)=\langle \langle \Xi\varphi_{\zeta},\varphi_{\eta}\rangle \rangle, \hspace{.1in}  \forall \zeta, \eta\in E_{\mC}.
\end{displaymath}

According to \cite{Berezin1966, Berezin1971, Obata1999, Ji2021}, for any two operators $\Xi_{1},\Xi_{2}\in {\mathcal L}((E), (E)^{*})$ there exists $\Xi\in {\mathcal L}((E), (E)^{*})$ uniquely determined by
\begin{displaymath}
	\hat{\Xi}(\zeta, \eta) =e^{-\langle \zeta,\eta\rangle}\hat{\Xi}_{1}(\zeta, \eta)\hat{\Xi}_{2}(\zeta, \eta) , \hspace{.1in} \forall \zeta, \eta \in E_{\mC}.
\end{displaymath}
This defines the operator Wick product $\Xi =\Xi_{1}\diamond \Xi_{2}.$  The operator Wick quantization of our nonlinearity is then defined by:
\begin{displaymath}
	\langle \langle J(\varphi)^{\diamond} \varphi_{\zeta}, \varphi_{\eta}\rangle \rangle = e^{-m\langle \zeta,\eta\rangle} J(\hat{\varphi}(\zeta,\eta)), 
\end{displaymath}
for some $m>0$. Here
\begin{displaymath}
	\hat{\varphi}(\zeta, \eta) =\langle \langle \varphi \varphi_{\zeta}, \varphi_{\eta}\rangle \rangle, \hspace{.1in}  \forall \zeta, \eta\in E_{\mC}.
\end{displaymath}

Let us recall the characterization theorem for symbols \cite{Obata1997, Obata1999}:

\begin{Thm} \label{Thm:3.5}
	Let $\Theta$ be a ${\mathcal L}(H)$-valued function on $E_{\mC}\times E_{\mC}$. Then it is the symbol of an operator in ${\mathcal L}((E)_{\beta}\otimes H, (E)^{*}_{\beta}\otimes H)$ if and only if
	\begin{enumerate}
		\item For any $\zeta, \zeta_{1},\eta. \eta_{1}\in E_{\mC}$, and $u,v\in H$ the function
		\begin{displaymath}
			z, w \rightarrow \langle \Theta(z\zeta+\zeta_{1}, w\eta+\eta_{1})u,v\rangle, \hspace{.1in} z, w\in \mC,
		\end{displaymath}
		is entire holomorphic on $\mC \times \mC$;
		\item There exists constants $C\geq 0, K\geq 0$ and $p\geq 0$ such that
		\begin{displaymath}
			\| \Theta (\zeta, \eta)\|_{{\mathcal L}(H)} \leq C \exp{\left ( K (\vert \zeta\vert_{p}^{\frac{2}{1-\beta}}+ \vert \eta\vert_{p}^{\frac{2}{1-\beta}})\right)}, \hspace{.1in} \zeta, \eta \in E_{\mC}.
		\end{displaymath}
	\end{enumerate}
\end{Thm}

Now we take the symbol of the equation \ref{eqn:1.4} and get equation \ref{eqn:2.1} for the symbols associated with the operator $\varphi$.	As a consequence of Theorem \ref{Thm:2.4}, Theorem \ref{Thm:3.4} and Theorem \ref{Thm:3.5} we state:

\begin{Pro}\label{Thm:3.6}
	Let $A$ and $J$ satisfy the Hypothesis \ref{Thm:2.1} and that of Theorem \ref{Thm:2.7}. For $\varphi_{0}\in {\mathcal L}(H)$ and small enough time $T$ there exists a unique solution $\varphi \in {\mathcal L}((E)_{\beta}\otimes H, (E)^{*}_{\beta}\otimes H)$ to \ref{eqn:1.4} such that the following operator series converges in $ ((E)_{\beta}\otimes H)^{*}$ for $\deg{\Xi} \leq 2/(1-\beta)$:
	\begin{displaymath}
		\varphi(\cdot)\psi=\sum_{l,m=0}^{\infty}\Xi_{l,m}(\kappa_{l,m})\psi, \hspace{.1in} \psi\in (E)_{\beta}\otimes H.
	\end{displaymath}
\end{Pro}
\begin{proof}
	We can take two approaches to this theorem. First one is to take the symbol of \ref{eqn:1.4}. Establish solvability for the nonlinear operator version of equation \ref{eqn:2.1} and then use the characterization theorem \ref{Thm:3.5} to complete the construction of the solution. We note here that the symbols of the equation \ref{eqn:2.1} satisfy an operator valued evolution and such operator evolutions are extensively studied in the context of infinite dimensional linear control systems and filtering theory \cite{Temam1971, DaPrato1972, Tartar1974, Kuiper1980, Kuiper1985}. These are called infinite dimensional Riccati equations and their generalizations. We will elaborate the use of these theories in the current context in our forthcoming papers \cite{Sritharan2025a, Sritharan2025b}.
	
	A second approach would be to use Yosida approximation to $A$ to obtain a bounded operator $A_{\lambda}=\lambda A (\lambda I-A)^{-1}, \lambda >0$, and a suitable $\mu$-ball truncation in ${\mathcal H}$ leading to a bounded nonlinearity $J_{\mu}$ for $J$ followed by Wick quantization to obtain $J_{\mu}^{\diamond}$ so that we end up with a quantum dynamic stochastic differential equation studied for example by \cite{Obata1997, Obata1999}:
	\begin{displaymath}
		\frac{d}{dt}\Phi = \Xi_{t} \diamond \Phi + M_{t}
	\end{displaymath}
	where
	\begin{displaymath}
		M_{t}=-i A_{\lambda} \diamond \Phi + J_{\mu}(\Phi)^{\diamond},
	\end{displaymath}
	and express the solution as \cite{Obata1999}:
	\begin{displaymath}
		\Phi_{t}= \mbox{wexp } \Omega_{t}\diamond \left (\int_{0}^{t} \mbox{wexp } (-\Omega_{s}) \diamond M_{s}ds + \Phi_{0}   \right)
	\end{displaymath}
	where
	\begin{displaymath}
		\Omega_{t}=\int_{0}^{t}\Xi_{s}ds \hspace{.1in} \mbox{  and  }\hspace{.1in}  \mbox{ wexp } \Omega = \sum_{n=0}^{\infty} \frac{1}{n!} \underbrace{\Omega\diamond\cdots \diamond \Omega}_{n  \small{\mbox{ times}}}.
	\end{displaymath}
	Proof would then be completed by taking $\lambda\rightarrow \infty$ and $\mu\rightarrow \infty$. These methods will be elaborated in the forthcoming papers \cite{Sritharan2025a, Sritharan2025b}.
	
\end{proof}

\section{Laser Propagation and Generation Models}
In this section we will give formal derivations of various Laser propagation and generation models and discuss the individual mathematical characteristics of each of the models.
\subsection{Continuous and Pulse Wave Models: The Stochastic Paraxial and Klein-Gordon Equations}

The paraxial equation is a very well-known model widely used in the literature on the acoustic wave and electromagnetic wave propagation in random media \cite{Tatarski1961, Papa1973, Tap1977, Strohbehn1978, Sp2002,Sprangle2003, Gustafsson2019, Sritharan2023}. As in these papers, we derive from the Maxwell's equations a wave equation for the electric field $E(x,t)$:
\begin{equation}
	\Delta E(x,t) - \frac{1}{c^{2}} \frac{\partial^{2}}{\partial t^{2}}	(n^{2}(x,t) E(x,t))= -2\nabla (E(x,t)\cdot\nabla \log (n(x,t))), \label{eqn2.1}	
\end{equation}
where $n(x,t)$ is the refractive index of the medium and $c$ is the speed of light. We assume that the time scale of fluctuations in the medium is much slower than the light speed and invoke further simplifications based on this assumption. Thus neglecting the right hand side and also the $n^{2}$ term out of the time derivative we arrive at
\begin{equation}
	\Delta E(x,t) - \frac{n^{2}(x,t)}{c^{2}} \frac{\partial^{2}}{\partial t^{2}}	 E(x,t)=0. \label{eqn2.2}
\end{equation}
Substituting a plane wave solution $E(x_{1},x_{2},x_{3},t)=\psi(x_{1},x_{2},x_{3})\exp(ik x_{3}-i\omega t)$ and neglecting the back-scatter term  $\frac{\partial^{2}\psi(x_{1},x_{2},x_{3})}{\partial x_{3}^{2}}$ (using simple scaling argument, see for example \cite{Papa1973, Strohbehn1978}) we arrive at the nonlinear random paraxial equation:
\begin{displaymath}
	i \frac{\partial \psi(x_{1},x_{2},x_{3})}{\partial x_{3}}+ \Delta_{\bot}\psi(x_{1},x_{2},x_{3}) +\chi(x_{1},x_{2},x_{3})\psi(x_{1},x_{2},x_{3})
\end{displaymath}
\begin{equation}
\pm
	\vert \psi(x_{1},x_{2},x_{3})\vert^{p-1}\psi(x_{1},x_{2},x_{3})=0. \label{eqn2.3}
\end{equation}
where $\Delta_{\bot} $ is the two dimensional Laplacian in the variables $x_{2},x_{3}$ and $\chi$ is a random field that depends on the medium and the field strength. 
Renaming the time-like variable $x_{3}$ as $t$ and suppressing the actual time variable $t$ we end up with the {\em two dimensional} linear or nonlinear stochastic Schr\"odinger equation:
\begin{displaymath}
	i \frac{\partial \psi(x_{1},x_{2},t)}{\partial t}+ \Delta_{\bot}\psi(x_{1},x_{2},t) +\chi(x_{1},x_{2},t)\psi(x_{1},x_{2},t)
\end{displaymath}
\begin{equation}
	\pm\vert \psi(x_{1},x_{2},t)\vert^{p-1}\psi(x_{1},x_{2},t)=0. \label{eqn2.4}
\end{equation}
On the other hand, using a wave envelop type representation 
\begin{displaymath}
	E(x_{1},x_{2},x_{3},t)=\psi(x_{1},x_{2},x_{3},t)\exp(ik x_{3}-i\omega t)
\end{displaymath}
in \ref{eqn2.1},we arrive at a time dependent random nonlinear wave equation \cite{Sprangle2003}:
\begin{equation}
	\frac{\partial^{2}}{\partial t^{2}}\psi(x,t)	-(\Delta-k_{0}^{2})\psi(x,t)+\chi(x,t)\psi(x,t)\pm\vert\psi (x,t)\vert^{p-1}\psi(x,t)=0,\label{eqn2.5}
\end{equation}
where $\chi$ is a random field that depend on the medium. Here we have also neglected certain first order derivative terms that appear in \cite{Sprangle2003} to arrive at nonlinear random Klein-Gordon type wave equation. Setting $v=\partial_{t}\psi$ we reframe the above dynamics as:
\begin{displaymath}
	\frac{\partial}{\partial t} \begin{bmatrix}
		\psi\\
		v
	\end{bmatrix}
	- \begin{bmatrix}
		0 & I\\
		-B^{2} & 0
	\end{bmatrix}
	\begin{bmatrix}
		\psi\\
		v
	\end{bmatrix}
	+
	\begin{bmatrix}
		0\\
		\pm 	\vert \psi\vert^{p-1}\psi
	\end{bmatrix}
	+
	\begin{bmatrix}
		0\\
		\chi \psi
	\end{bmatrix}
	=0,
\end{displaymath}
where $B=\sqrt{ -\Delta +k_{0}^{2}I}$.

\begin{displaymath}
	H= D(B)\oplus L^{2}(\mR^{3}) \mbox{ and } D(A)= D(B^{2})\oplus D(B).
\end{displaymath}

Then we have the following estimate \cite{Reed1976}:

\begin{Lem}
	Setting $\varphi =\langle \psi, v\rangle$ and $J(\varphi)=\langle 0, \vert \psi\vert^{2}\psi\rangle$, for $\varphi_{1},\varphi_{2}\in H$, $J$ satisfies
	\begin{enumerate}
		\item $\|J(\varphi_{1})\|\leq K\|\varphi_{1}\|^{3}$,
		\item $\|J(\varphi_{1})-J(\varphi_{2})\|\leq C(\|\varphi_{1}\|,\|\varphi_{2}\|) \|\varphi_{1}-\varphi_{2}\|$,
		\item $\|AJ(\varphi_{1})\|\leq K\|\varphi_{1}\|^{2}\|A \varphi_{1}\|$,
		\item $\|A(J(\varphi_{1})-J(\varphi_{2}))\|\leq C(\|\varphi_{1}\|,\|\varphi_{2}\|,\|A \varphi_{1}\|,\|A \varphi_{2}\|) \|A\varphi_{1}-A\varphi_{2}\|$.
	\end{enumerate}
\end{Lem}

\subsection{Zakharov System and Langmuir Turbulence in Plasma}
The Zakharov equation \cite{Zakharov1972} for electromagnetic wave interaction with plasma can be derived using two-fluid model for ion-electron dynamics coupled with the Maxwell's equations as in  \cite{Sulem1999} is given by:
\begin{displaymath}
	i\frac{\partial \psi}{\partial t} +\Delta \psi - n \psi =0,
\end{displaymath}
\begin{displaymath}
	\frac{\partial^{2}n}{\partial t^{2}}-\Delta n=\Delta \vert \psi\vert^{2}.
\end{displaymath}
A number of rigorous studies have been made \cite{Added1988, Sulem1999, Ginibre1997} and a typical local solvability established in these papers is as follows:

\begin{Thm}
For $d\geq 2$, and $(\psi(\cdot, 0), n(\cdot, 0), \frac{\partial }{\partial t}n(\cdot, 0))	\in H^{k}(\mR^{d})\times H^{l}(\mR^{d})\times H^{l-1}(\mR^{d})$, $l\geq 0, k\geq\frac{1}{2}(l+1)$, there exists a unique solution 
\begin{displaymath}
(\psi, n, \partial_{t}n)\in C([0,T^{*}); H^{k}(\mR^{d})\times H^{l}(\mR^{d})\times H^{l-1}(\mR^{d})),
\end{displaymath}
 where $T^{*}$ depends on the initial data. For $d=1$ such a solution is global in positive and negative time $t\in \mR$.
\end{Thm}

We can however frame the Zakharov equation in our unified formulation by the following change of variables.

Denoting $\vert \nabla \vert =\sqrt{-\Delta}$ and using the substitution
\begin{displaymath}
	v=n-i\vert\nabla\vert^{-1}\partial_{t}n
\end{displaymath}
in the Zakharov system we arrive at
\begin{equation}
	i\frac{\partial \psi}{\partial t} +\Delta \psi - \psi \Re{(v)}=0,	\label{eqn:2.6}
\end{equation}
\begin{equation}
	i\frac{\partial v}{\partial t}+\vert \nabla \vert v +\vert \nabla \vert \vert \psi\vert^{2}=0.\label{eqn:2.7}
\end{equation}
We write this system in the semilinear form as
\begin{equation}
	i	\frac{\partial}{\partial t} \begin{bmatrix}
		\psi\\
		v
	\end{bmatrix}
	+ \begin{bmatrix}
		\Delta & 0\\
		0 & \vert \nabla \vert
	\end{bmatrix}
	\begin{bmatrix}
		\psi\\
		v
	\end{bmatrix}
	+
	\begin{bmatrix}
		- \psi \Re{(v)}\\
		\vert \nabla \vert \vert \psi\vert^{2}
	\end{bmatrix}
	=0. \label{eqn:2.8}
\end{equation}
We can take $H=H^{3}(\mR^{3})\oplus H^{2}(\mR^{3})$ see for example \cite{Sulem1999}.

\subsection{Free Electron Laser Models: Maxwell-Dirac Equations}
Maxwell-Dirac equations describing classical quantum electrodynamics can also be regarded as the fundamental governing equations for free electron lasers \cite{Madey2010, Becker1979} and high power microwave generation dynamics. Rigorous theory of these models in the detrministic setting has a long history \cite{Dyson1949, Gross1966, Chadam1973, Flato1987, Strauss1978}. We will give a brief overview of this model and indicate how it fits in the unified mathematical theory and stochastic analysis developed in this paper. We start with the Maxwell's equations:
\begin{displaymath}
	\nabla \cdot E =\rho, \hspace{.1in} \nabla \cdot B=0,
\end{displaymath}
\begin{displaymath}
	\nabla \times E +\partial_{t}B=0, \hspace{.1in} \nabla \times B -\partial_{t}E =J.
\end{displaymath}
Dirac equation
\begin{displaymath}
	\left ( \alpha^{\mu}D_{\mu}+ m\beta\right)\psi=0.
\end{displaymath}
Here $m\geq 0$, $\alpha^{\mu}, \beta$ are $4\times 4 $ Dirac matrices. 

$E,B:\mR^{1+3}\rightarrow \mR$ are electric field and magnetic field respectively and $\psi:\mR^{1+3}\rightarrow \mC^{4}$ is the Dirac spinner field. 
Here $\alpha $ are linear operators in spin space that satisfy $\alpha^{\mu}\alpha^{\nu}+\alpha^{\nu}\alpha^{\mu}=2 g^{\mu \nu}$. $g^{00}=1, g^{11}=-1, g^{\mu\nu}=0$ for $\nu \not = \mu$ and $\alpha^{0*}=\alpha^{0}, \alpha^{1*}=-\alpha^{1}$.

We represent the electromagnetic fields by real four dimensional vector potential $A_{\mu}$, $\mu=0,1,2,3$:
\begin{displaymath}
	B=\nabla A, \hspace{.1in} E=\nabla A_{0}-\partial_{t}A, \mbox{ with  } A=(A_{1},A_{2}, A_{3}).
\end{displaymath}
The couplings are:
\begin{displaymath}
	J^{\mu}=\langle \alpha^{\mu}\psi,\psi\rangle_{\mC^{4}},
\end{displaymath}
\begin{displaymath}
	\rho=J^{0}=\vert \psi\vert^{2},\hspace{.1in} J=(J^{1}, J^{2}, J^{3}),
\end{displaymath}
the gauge covariant derivative
\begin{displaymath}
	D_{\mu}= \partial^{(A)}_{\mu}=\frac{1}{i}\partial_{\mu} -A_{\mu}.
\end{displaymath}
The above system is invariant under the transform:
\begin{displaymath}
	\psi\rightarrow \psi'=e^{i\chi}\psi, \hspace{.1in} A_{\mu}\rightarrow A'_{\mu}=A_{\mu}+\partial_{\mu}\chi,
\end{displaymath}
for any $\chi:\mR^{1+3}\rightarrow \mR$. 

We impose Lorenz gauge condition:
\begin{displaymath}
	\partial^{\mu}A_{\mu}=0    \mbox{  which is the same as   } \partial_{0}A_{0}=\nabla\cdot A.
\end{displaymath}
This results in the Maxwell-Dirac system:
\begin{equation}
	\square A_{\mu} =(\Delta -\partial_{0}^{2})A_{\mu} =  -\langle \alpha_{\mu}\psi, \psi\rangle_{\C^{4}}, \label{eqn:2.9}
\end{equation}

\begin{equation}
	\left ( -i \alpha^{\mu}\partial_{\mu} + m\beta \right )\psi =  A_{\mu}\alpha^{\mu} \psi, \label{eqn:2.10}
\end{equation}
along with the Lorenz gauge condition.
Now denoting the operator:
\begin{displaymath}
	A_{W}=- i\begin{bmatrix}
		0 & I\\
		-B_{0}^{2} & 0
	\end{bmatrix}
\end{displaymath}
where $B_{0}=\sqrt{-\Delta+k_{0}^{2}I}$.

The Dirac operator as $D_{e}=-i \alpha\cdot\nabla + m\beta $, 
\begin{displaymath}
	D_{e}^{2}=\left(
	\begin{array}{ccccc}
		B_{e}^{2}                                    \\
		& B_{e}^{2}              &   & \text{\huge0}\\
		&               & B_{e}^{2}                 \\
		& \text{\huge0} &   & B_{e}^{2}            \\
	\end{array}
	\right),
\end{displaymath}
where $B_{e}=\sqrt{-\Delta+m^{2}I}$.

Now setting $v=(A, \partial_{t}A)$ we can write the Maxwell-Dirac system as a semilinear evolution:
\begin{equation}
	\frac{\partial}{\partial t} \begin{bmatrix}
		\psi\\
		v
	\end{bmatrix}
	+ i\begin{bmatrix}
		D_{e} & 0\\
		0 & A_{W}
	\end{bmatrix}
	\begin{bmatrix}
		\psi\\
		v
	\end{bmatrix}
	+
	\begin{bmatrix}
		-	A_{\mu}\alpha^{\mu} \psi\\
		(0,i\langle \alpha_{\mu}\psi, \psi\rangle_{\mC^{4}} +k_{0}^{2})
	\end{bmatrix}
	=0. \label{eqn:2.11}
\end{equation}
We will take escalated energy spaces \cite{Chadam1972, Chadam1973, Reed1976} as state space:
\begin{displaymath}
	H= (\oplus_{i=0}^{3}D(B_{e}^{3}))\oplus (D(B_{0}^{3}) \oplus D(B_{0}^{2})).
\end{displaymath}
Then the operator:
\begin{displaymath}
	A=
	\begin{bmatrix}
		D_{e} & 0\\
		0 & A_{W}
	\end{bmatrix}	
\end{displaymath}
is self-adjoint on
\begin{displaymath}
	D(A)= (\oplus_{i=0}^{3}D(B_{e}^{4}))\oplus (D(B_{0}^{4}) \oplus D(B_{0}^{3})).
\end{displaymath}
The properties of the source term $J$ as required in the Hypothesis \ref{Hyp:2.1} and the main theorems are verified in \cite{Chadam1972, Chadam1973, Chadam1973b,Chadam1974, Reed1976}.

\subsection{Sine-Gordon Equation}
Sine-Gordon equation is a well-known model in quantum field theory and soliton theory. In \cite{Sritharan2023} we have initiated the stochastic quantization of this model and here we indicate that the hypotheses (Hypothesis \ref{Hyp:2.1} and hypothesis in Theorem \ref{Thm:2.7}) needed for the main theorems of this paper are satisfied by this model to ensure global unique solvability theorems \ref{Thm:2.5}, \ref{Thm:2.7} in Section 2.
\begin{displaymath}
	\frac{\partial^{2}}{\partial t^{2}}u-\Delta u +k_{0}^{2}u -g \mbox{ sin} (u) +V(t)u=0.
\end{displaymath}
Setting $v=\partial_{t}u$ and denoting $B=\sqrt{ -\Delta +k_{0}^{2}I}$ we reframe the above dynamics as:
\begin{displaymath}
	\frac{\partial}{\partial t} \begin{bmatrix}
		u\\
		v
	\end{bmatrix}
	- \begin{bmatrix}
		0 & I\\
		-B^{2} & 0
	\end{bmatrix}
	\begin{bmatrix}
		u\\
		v
	\end{bmatrix}
	+
	\begin{bmatrix}
		0\\
		g \mbox{ sin} (u)
	\end{bmatrix}
	+
	\begin{bmatrix}
		0\\
		V(t)u
	\end{bmatrix}
	=0,
\end{displaymath}
which is in the form \ref{eqn:1.1}. Let $H=D(B)\oplus L^{2}(\mR^{n})$, and $D(A)=D(B^{2})\oplus D(B)$. We then have \cite{Reed1976}:

\begin{Lem} Setting $ \varphi=\langle u, v\rangle$, $J(\varphi)=\langle 0, \mbox{ sin} (u)\rangle$ we have the following estimates: 
	\begin{enumerate}
		\item $\|J(\varphi)\|\leq \|\varphi\|,$
		\item $\|AJ(\varphi)\|^{2}\leq K \|\varphi\|^{2}$,
		\item $\|J(\varphi)-J(\psi)\|\leq K \|\varphi-\psi\|$,
		\item $ \|A(J(\varphi)-J(\psi))\|\leq K \|\varphi-\psi\|\|A\varphi\|+\|\varphi-\psi\|$.
	\end{enumerate}
\end{Lem}


\begin{thebibliography}{30} 
	
	\bibitem{Added1988} H. Added and S. Added, Equations of Langmuir turbulence and nonlinear Schr\"odinger equation: smoothness and approximation, {\it J. Functional Anal.}, {\bf 79}, (1988), 183-210.
	\bibitem{Alshanskiy2014} M. A. Alshanskiy, Wiener-Itô chaos expansion of Hilbert space valued random variables, {\it Journal of Probability} {\bf 2014}, Article ID 786854, (2014).
	\bibitem{Barbu1976} V. Barbu, {\it Nonlinear Semigroups and Differential Equations in Banach Spaces}, Nordorff Publishers, Netherland, 1976.
	\bibitem{Becker1979} W. Becker and H. Mitter, Quantum theory of a free electron laser, {\it Z. Physik B}, {\bf 35}, (1979), 399-404.
	\bibitem{Berezin1966} F. A. Berezin, {\it The Method of Second Quantization}, Academic Press, New York, 1966.
	\bibitem{Berezin1971} F. A. Berezin, Wick and anti-Wick operator symbols, {\it Math. USSR-Sb.,} {\bf 15}, (1971), 577-606.
	\bibitem{Cazenave1998} T. Cazenave and A. Haraux, {\it An Introduction to Semilinear Evolution Equations}, Clarendon Press, Oxford, 1998.
	\bibitem{Chadam1972} J. M. Chadam, On the Cauchy problem for the coupled Maxwell-Dirac equations, {J. Math. Phys.}, {\bf  13}, (1972), 597-604.
	\bibitem{Chadam1973} J. M. Chadam, Global solutions of the Cauchy problem for the (classical) coupled Maxwell-Dirac equations in one space dimension, {\it Journal of Functional Analysis}, {\bf 13}, (1973), 173-184.
	\bibitem{Chadam1973b} J. M. Chadam, Asymptotic behavior of equations arising in quantum field theory, {\it Applicab!e Analysis}, {\bf 3}, (1973), 377-402.
	\bibitem{Chadam1974} J. M. Chadam and R. T. Glassey, On certain global solutions for the Cauchy problem for (classical) coupled Klein-Gordon-Dirac equations in one and three space dimensions, {Archiv for Rational Mechanics and Analysis}, {\bf 54}, (1974), Issue.3, 223-237.
	\bibitem{DaPrato1972} G. Da Prato, Quelques r\'esultats d'existence et r\'egularit\'e pour un probl\'eme non lin\'ire de la th\'eorie du contr\^ole, {\it  Bull. Soc. math. France}, {\bf  31-32}, (1972), 127-132. 
	\bibitem{DaPrato1996} G. Da Prato and J. Zabczyk, {\it Ergodicity for Infinite Dimensional Systems}, Cambridge University Press, 1996.
	\bibitem{DaPrato2014} G. Da Prato and J. Zabczyk, {\it Stochastic Equations in Infinite Dimensions}, Cambridge University Press, 2014.
	\bibitem{Derm1996} A. Dermoune, Non-commutative Burgers equation, {\it Hokkaido mathematical journal}, {\bf 25}, (1996), 315-332.
	\bibitem{Dyson1949} F. J. Dyson, The radiation theories of Tomonaga, Schwinger, and Feynman, {\it Physical Review}, {\bf 75}, No. 3, (1949), 486-502. 
	\bibitem{Flato1987} M. Flato, J. Simon and E. Taflin, On global solutions of the Maxwell-Dirac equations, {\it Commun. Math. Phys.} {\bf 112}, (1987), 21-49.
	\bibitem{Gelfand1964} I. M. Gelfand and N. Ya. Vilenkin, {\it Generalized Functions}, {\bf 4}, Academic Press,New York, 1964.
	\bibitem{Ginibre1997} J. Ginibre, Y. Tsutsumi and G. Velo, On the Cauchy problem for the Zakharov system, {\it J. Functional Analysis}, Vol. 151, (1997), 384-436.
	\bibitem{Gross1966} L. Gross, The Cauchy problem for the coupled Maxwell and Dirac equations, {\it Communications on Pure and Applied Mathematics}, {\bf  XIX}, (1966), 1-15.
	\bibitem{Gustafsson2019} Jonathan Gustafsson, Benjamin F. Akers, Jonah A. Reeger, Sivaguru S. Sritharan, Atmospheric propagation of high energy lasers, {\it Eng. Math. Lett.}, {\bf 2019} (2019), Article ID 7.
	\bibitem{Hida1980} T. Hida, {\it Brownian Motion}, Springer-Verlag, New York, 1980.
	\bibitem{Hida1994} T. Hida, H-H. Kuo, J. Potthoff and L. Streit, {\it White Noise: An Infinite Dimensional Calculus}, Springer, New York, 1994.
	\bibitem{Hille1996} E. Hille and R. S. Phillips, {\it Functional Analysis and Semigroups}, American Mathematical Society, Providence, RI., 1996.
	\bibitem{Holden2010} H. Holden, B. Oksendal, J. Uboe and T. Zhang, {\it Stochastic Partial Differential Equations: A Modeling, White Noise Functional Approach}, Springer-Verlag, New York, 2010.
	\bibitem{Ji2021} U. C. Ji and P. C. Ma, Wick calculus for vector-valued Gaussian white noise functionals, {\it Probability and Mathematical Statistics}, {\bf 41},  (2021), No.2, 283–302.
	\bibitem{Kato1976} T. Kato, {\it Perturbation Theory of Linear Operators}, Springer-Verlag, Berlin, 1976.
	\bibitem{Kato1976a} T. Kato, Linear and quasi-linear equations of evolutions of hyperbolic type, in: {\it Hyperbolicity}, Edited by G. Da Prato and G. Geymonat, CIME Summer School, {\bf 72},(1976),125-191, Springer-Verlag, Berlin.
	\bibitem{Kon1993} Yu. G. Kondratiev and L. Streit, Spaces of white noise distributions: Constructions, descriptions, applications, I, {\it Rep. Math. Phys.}, {\bf 33}, (1993), 341-366.
	\bibitem{Kree1978} P. Kre\'{e} and R. Raczka, Kernels and symbols of operators in quantum field theory, {\it Ann. Inst. H. Poincar\'{e}}, {\bf A28}, (1978), 41-73.
	\bibitem{Kuiper1980} H. J. Kupier and S. M. Shew, Strong solutions for infinite dimensional Riccati equations arising in transport theory, {\it SIAM J. Math. Anal.}, {\bf 11}, (1980), No.2, 211-222.
	\bibitem{Kuiper1985} H. J. Kupier, Generalized operator Riccati equations, {SIAM J. Math. Anal.}, {\bf 16}, (1985), No. 4, 675-694.
	\bibitem{Kubo1980} I. Kubo and S. Takenaka, Calculus on Gaussian white noise I-IV, {\it Proc. Japan Acad.}, {\bf 56A} (1980), 376-380; 411-416; {\bf 57A} (1981), 433-437; {\bf 58A} (1982), 186-189.
	\bibitem{Kuo1975} H-H, Kuo, {\it Gaussian Measures in Banach Spaces }, Lecture Notes in Mathematics, Vol. 463, Springer-Verlag, New York, 1975.
	\bibitem{Kuo1996} H-H, Kuo, {\it White Noise Distribution Theory}, CRC Press, Boca Raton, Fl., 1996.
	\bibitem{Madey2010} J. M. Madey, Invention of the Free Electron Laser, {\it Reviews of Accelerator Science and Technology, } {\bf  3}, (2010) 1–12.
	\bibitem{Meyer1973} G. H. Meyer, Edited, {Initial Value Methods for Boundary Value Problems: Theory and Applications of Invariant Embedding}, Academic Press, New York, 1973.
	\bibitem{Mohan2017} M. T. Mohan and S. S. Sritharan, $L^{p}$-solutions to stochastic Navier-Stokes equation with L\'evy noise with $L^{m}$ initial data, {\it Evolution Equation and Control Theory}, {\bf 6},(2017), No. 3, 409-425.
	\bibitem{Mohan2018} M. T. Mohan and S. S. Sritharan, Stochastic quasilinear symmetric hyperbolic system perturbed by L\'evy noise, {\it Journal of Pure and Applied Functional Analysis}, {\bf 3}, (2018), No.1,  137-178.
	\bibitem{Obata1993} N. Obata, An analytic characterization of symbols of operators on white noise functionals, {\it J. Math. Soc. Japan}, {\bf 45}, (1993),No.3, 421-445.
	\bibitem{Obata1994} N. Obata, {\it White Noise Calculus and Fock Space}, Springer-Verlag, New York, 1994.
	\bibitem{Obata1994a} N. Obata, Operator calculus on vector-valued white noise functionals, {\it J. Functional Analysis}, {\bf 121}, (1994), 185-232.
	\bibitem{Obata1997} N. Obata, Quantum stochastic differential equations in terms of quantum white noise, {\it Nonlinear Analysis, Theory, Methods and Applications}, {\bf 30}, (1997),No.1, 279-290.
	\bibitem{Obata1999} N. Obata, Wick product of white noise operators and quantum stochastic differential equations, {\it J. Math. Soc. Japan}, {\bf 51}, (1999), No.3, 613-641.
	\bibitem{Ouerdiane1998} H. Ouerdiane, Algébres nucléaires de fonctions entiéres et équations aux derivées partielles stochastiques, {\it Nagoya Mathematical Journal}, {\bf 151}, (1998), 107-127.
	\bibitem{Papa1973} G. C. Papanicolaou, D. McLaughlin, and R. Burridge, A Stochastic Gaussian Beam, {\it J. Math, Phys.}, {\bf 14}, (1973), No.1, 84-89.
	\bibitem{Parisi1981} G. Parisi and Y. S. Wu, Perturbation theory without gauge fixing, {Scientia Sinica}, {\bf XXIV}, (1981),No.4, 483-496.
	\bibitem{Pazy1983} A. Pazy, {\it Semigroups of Linear Operators and Applications to Partial Differential Equations}, Springer-Verlag, New York, 1983.
	\bibitem{Reed1975} M. C. Reed, Higher order estimates and smoothness of nonlinear wave equations, {\it Proceedings of the American Mathematical Society}, {\bf  51}, (1975),No.1, 79-85.
	\bibitem{Reed1976} M. C. Reed, {\it Abstract Nonlinear Wave Equations}, Lecture Notes in Mathematics, {\bf 507}, Springer-Verlag, Berlin 1976.
	\bibitem{Segal1963} I. Segal, Nonlinear Semigroups, {\it Annals of Mathematics}, {\bf 78}, (1963),No.2, 339-363.
	\bibitem{Simon1974} B. Simon, {\it The $P(\Phi)_{2}$ Euclidean (Quantum) Field Theory}, Princeton Series in Physics, Princeton, NJ, 1974.
	\bibitem{Sp2002} P. Sprangle, J. R. Peñano, and B. Hafizi, 	Propagation of intense short laser pulses in the atmosphere, {\it Phys. Rev.} {E \bf 66}, Issue: 4, 2002, 1-21.
	\bibitem{Sprangle2003} P. Sprangle, J. P. Pe\~nano, A. Ting, B. Hafizi and D. E. Gordon, Propagation of short, High-intensity Laser pulses in air, {\it Journal of Directed Energy}, {\bf 1}, (2003), 73-92.
	\bibitem{Sritharan2023} S. S. Sritharan and S. Mudaliar, Stochastic quantization of Laser propagation models, {\it Infinite Dimensional Analysis, Quantum Probability and Related Topics}, Published on-line July 2023.
	\bibitem{Sritharan2025a} S. S. Sritharan and S. Mudaliar, Approximations of stochastic nonlinear wave equations, {in preparation} (2025).
	\bibitem{Sritharan2025b} S. S. Sritharan and S. Mudaliar, Filtering of stochastic nonlinear wave equations, {in preparation} (2025).
	\bibitem{Strauss1989} W. Strauss, {\it Nonlinear Wave Equations}, CBMS Regional Conference Series No.73, American Mathematical Society, Providence,1989.
	\bibitem{Strauss1978} W. Strauss, Nonlinear invariant wave equations, {\it Invariant Wave Equations}, Lecture notes in Physics, No. 78,  (1978), 197-249, Springer-Verlag, New York.
	\bibitem{Strohbehn1978} J. W. Strohbehn, {\it Laser Beam Propagation in the Atmosphere}, Springer-Verlag, New York, 1978.
	\bibitem{Sulem1999} C. Sulem and P-L. Sulem, {\it Nonlinear Schr\"odinger Equations: Self-Focusing and Wave Collapse}, Springer-Verlag, New York 1999.
	\bibitem{Tap1977} F. D. Tappert, The parabolic approximation method, in  Joseph B. Keller, John S. Papadakis, Editors, {\it Wave Propagation and Underwater Acoustics}, Springer-Verlag, Berlin, 1977.
	\bibitem{Tartar1974} L. Tartar, Sur I’\'Etude Directe d'Equations non Lin\'eaires lntervenant en Th\'eorie du Contr\^ole Optimal, {J. Functional Analysis}, {\bf 6}, (1974), 1-47.
	\bibitem{Temam1971} R. Temam, Sur I'\'equation de Riccati associ\'ee \'a des op\'erateurs non born\'es, en dimension infinie,{J. Functional Analysis}, {\bf 7}, (1971), 85-115.
	\bibitem{vanNeerven2020} J. van Neerven, M. Veraar, Maximal inequalities for stochastic convolutions in 2-smooth Banach spaces and applications to stochastic evolution equations, {\it Phil. Trans. R. Soc. A} {\bf 378}: 20190622.
	\bibitem{Tatarski1961} V. I. Tatarski, {\it Wave Propagation in a Turbulent Medium}, Dover Publishers, New York, 1961.
	\bibitem{Yosida1991} K. Yosida, {\it Functional Analysis}, Springer-Verlag, New York, (1991).
	\bibitem{Zakharov1972} V. E. Zakharov, Collapse of Langmuir waves, {\it Soviet Phys. JETP}, {\bf 35}, (1972), No.5, 908-913.
	\bibitem{Zohuri2016} B. Zohuri, {\it Directed Energy Weapons: Physics of High Energy Lasers (HEL)}, Springer-verlag, New York, 2016.
	
\end{thebibliography}
\end{document}